\documentclass[11pt]{article}
\usepackage{amsmath,amssymb,amsthm,color}

\usepackage{comment}
\usepackage{color}
\usepackage{amssymb}
\usepackage{amsmath}
\usepackage[english]{babel}
\usepackage{fancyhdr}
\usepackage{amsmath}

\def\ci{\begin{color}{red}\,}
\def\cf{\end{color}\,}

\newtheorem{thm}{Theorem}[section]
\newtheorem{corollary}[thm]{Corollary}
\newtheorem{lemma}[thm]{Lemma}
\newtheorem{proposition}[thm]{Proposition}
\newtheorem{definition}[thm]{Definition}
\newtheorem{theorem}[thm]{Theorem}
\newtheorem{remark}{Remark}[section]

\def\D{{\mathfrak D}\, }
\def\R{{\mathfrak R}\, }

\begin{document}
\begin{center}
{\bf Lie maps on alternative rings preserving idempotents}\footnote{
This work was supported by  FAPESP 19/03655-4; CNPq 302980/2019-9;  RFBR 20-01-00030. }

\vspace{.2in}
{\bf Bruno Leonardo Macedo Ferreira}
\\
Federal Technological University of Paran\'{a},\\
Professora Laura Pacheco Bastos Avenue, 800,\\
85053-510, Guarapuava, Brazil.
\\
brunoferreira@utfpr.edu.br
\\
\vspace{.2in}

{\bf Henrique Guzzo Jr.}
\\
University of S\~{a}o Paulo,\\
Mat\~{a}o Street, 1010,\\
05508-090, S\~{a}o Paulo, Brazil.
\\

guzzo@ime.usp.br
\\
\vspace{.2in}

{\bf Ivan Kaygorodov}
\\
Federal University of ABC,\\
dos Estados Avenue, 5001,\\
09210-580, Santo Andr\'{e}, Brazil. 
\\
kaygorodov.ivan@gmail.com
\\
\vspace{.2in}

\end{center}


\

AMS: 17A36, 17D05.

\begin{abstract}
Let $\R$ and $\R'$ unital $2$,$3$-torsion free alternative rings and $\varphi: \R \rightarrow \R'$ be a surjective Lie multiplicative map that preserves idempotents. Assume that $\R$ has a nontrivial idempotents. Under certain assumptions on $\R$, we prove that $\varphi$ is of the form $\psi + \tau$, where $\psi$ is either an isomorphism or the negative of an anti-isomorphism of $\R$ onto $\R'$ and $\tau$ is an additive mapping of $\R$ into the centre of $\R'$ which maps commutators into zero.
\end{abstract}

\vspace{0.5 in}
 \newpage
\section{Alternative rings and Lie multiplicative maps}

Let $\mathfrak{R}$ be a unital ring not necessarily associative or commutative and consider the following convention for its multiplication operation: $xy\cdot z = (xy)z$ and $x\cdot yz = x(yz)$ for $x,y,z\in \mathfrak{R}$, to reduce the number of parentheses. We denote the {\it associator} of $\mathfrak{R}$ by $(x,y,z)=xy\cdot z-x\cdot yz$ for $x,y,z\in \mathfrak{R}$. And $[x,y] = xy - yx$ is the usual Lie product of $x$ and $y$, with $x,y \in \mathfrak{R}$.

Let $\mathfrak{R}$ and $\mathfrak{R}'$ be two rings and $\varphi:\mathfrak{R}\rightarrow \mathfrak{R}'$ a map of $\mathfrak{R}$ into $\mathfrak{R}'$. We call $\varphi$ a {\it Lie multiplicative map} of $\mathfrak{R}$ into $\mathfrak{R}'$ if for all $x, y \in \mathfrak{R}$ 
\begin{eqnarray*}
\varphi\big([x,y]) = [\varphi(x),\varphi(y)].
\end{eqnarray*}
When $\varphi$ is an additive and bijective map we call {\it Lie isomorphism}. The study of Lie isomorphisms of rings was originally inspired by the work done by I. N. Herstein on generalizing classical theorems on the Lie structure of total matrix rings to results on the Lie structure of arbitrary simple rings.
In \cite{Mart}, W.S. Martindale studied Lie isomorphisms of a primitive ring $\R$ onto a primitive ring $\R'$, where he assumed that the characteristic of $\R$ is different from $2$ and $3$ and that $\R$ contained three nonzero orthogonal idempotents whose sum was the identity. A few years later \cite{Mart2}, he studied Lie isomorphisms of a simple ring $\R$ onto a simple ring $\R'$.    

A ring $\mathfrak{R}$ is said to be {\it alternative} if 
\[(x,x,y)=0 \mbox{ and }(y,x,x)=0, \mbox{ for all }x,y\in \mathfrak{R}.\] 

One easily sees that any associative ring is an alternative ring.

A ring $\mathfrak{R}$ is called {\it k-torsion free} if $k\,x=0$ implies $x=0,$ for any $x\in \mathfrak{R},$ where $k\in{\mathbb Z},\, k>0$, and {\it prime} if $\mathfrak{AB} \neq 0$ for any two nonzero ideals $\mathfrak{A},\mathfrak{B}\subseteq \mathfrak{R}$.
The {\it nucleus} of a ring $\mathfrak{R}$ \mbox{is defined by} $$\mathcal{N}(\mathfrak{R})=\{r\in \mathfrak{R}\mid (x,y,r)=0=(x,r,y)=(r,x,y) \hbox{ for all }x,y\in \mathfrak{R}\}.$$
The {\it commutative centre} of a ring $\mathfrak{R}$ \mbox{is defined by} $$\mathcal{Z}(\mathfrak{R})=\{r\in \mathcal{\R}\mid [r, x] = 0 \hbox{ for all }x \in \mathfrak{R}\}.$$
We will say that a map $\varphi: \R \rightarrow \R'$ \textit{preserves idempotents} if satisfies
$$e - \lambda f \in I(\R) \ \ \textrm{ if and only if } \ \ \varphi(e) - \lambda \varphi(f) \in I(\R')$$
where $I(\R)$ and $I(\R')$ are the set of all idempotents in $\R$ and $\R'$ respectively, $\lambda \in \mathbb{Q}$.

The next result can be found in \cite{bruth}
\begin{theorem}\label{meu}
Let $\mathfrak{R}$ be a $3$-torsion free alternative ring. So
$\mathfrak{R}$ is a prime ring if and only if $a\mathfrak{R} \cdot b=0$ (or $a \cdot \mathfrak{R}b=0$) implies $a = 0$ or $b =0$ for $a, b \in \mathfrak{R}$. 
\end{theorem}

\begin{definition}
A ring $\R$ is said to be flexible if satisfies
$$xy \cdot x = x \cdot yx \ \ for \ all \ x,y \in \R.$$
It is known that all alternative rings are flexible.
\end{definition}

\vspace{.2in}
A nonzero element $e_{1}\in \mathfrak{R}$ is called an {\it idempotent} if $e_{1}e_{1}=e_{1}$ and a {\it nontrivial idempotent} if it is an idempotent different from the multiplicative identity element of $\mathfrak{R}$. Let us consider $\mathfrak{R}$ an alternative ring and fix a nontrivial idempotent $e_{1}\in\mathfrak{R}$. Let \mbox{$e_2 \colon\mathfrak{R}\rightarrow\mathfrak{R}$} and $e'_2 \colon\mathfrak{R}\rightarrow\mathfrak{R}$ be linear operators given by $e_2(a)=a-e_1a$ and $e_2'(a)=a-ae_1.$ Clearly $e_2^2=e_2,$ $(e_2')^2=e_2'$ and we note that if $\mathfrak{R}$ has a unity, then we can consider $e_2=1-e_1\in \mathfrak{R}$. Let us denote $e_2(a)$ by $e_2a$ and $e_2'(a)$ by $ae_2$. It is easy to see that $e_ia\cdot e_j=e_i\cdot ae_j~(i,j=1,2)$ for all $a\in \mathfrak{R}$. Then $\mathfrak{R}$ has a Peirce decomposition
$\mathfrak{R}=\mathfrak{R}_{11}\oplus \mathfrak{R}_{12}\oplus
\mathfrak{R}_{21}\oplus \mathfrak{R}_{22},$ where
$\mathfrak{R}_{ij}=e_{i}\mathfrak{R}e_{j}$ $(i,j=1,2)$ \cite{He}, satisfying the following multiplicative relations:
\begin{enumerate}\label{asquatro}
\item [\it (i)] $\mathfrak{R}_{ij}\mathfrak{R}_{jl}\subseteq\mathfrak{R}_{il}\
(i,j,l=1,2);$
\item [\it (ii)] $\mathfrak{R}_{ij}\mathfrak{R}_{ij}\subseteq \mathfrak{R}_{ji}\
(i,j=1,2);$
\item [\it (iii)] $\mathfrak{R}_{ij}\mathfrak{R}_{kl}=0,$ if $j\neq k$ and
$(i,j)\neq (k,l),\ (i,j,k,l=1,2);$
\item [\it (iv.a)] $x_{ij}^{2}=0,$ for all $x_{ij}\in \mathfrak{R}_{ij}\ (i,j=1,2;~i\neq j);$

\item [\it (iv.b)] $x_{ij}y_{ij}=-y_{ij}x_{ij},$ for all $x_{ij},y_{ij}\in \mathfrak{R}_{ij}\ (i,j=1,2;~i\neq j).$

\end{enumerate}

The first result about the additivity of mappings on rings was given by Martindale III in \cite{Mart}, 
he established a condition on a ring $\mathfrak{R}$ such that every multiplicative isomorphism on $\mathfrak{R}$ 
is additive. In \cite{chang, changd}, Li and his coauthors also considered the 
almost additivity of maps for the case of Lie multiplicative mappings. They proved 

\begin{theorem}\label{cha}
Let $\mathfrak{R}$ be an associative ring containing a nontrivial idempotent $e_1$ and satisfying the
following condition: 
\[ (\mathbb{Q}) \mbox{  If }A_{11}B_{12} = B_{12}A_{22} \mbox{ for all }B_{12} \in \mathfrak{R}_{12},\] 
\noindent then $A_{11} + A_{22} \in \mathcal{Z}(\mathfrak{R})$. Let $\mathfrak{R}'$ be another ring. Suppose that a bijection map $\Phi\colon \mathfrak{R} \rightarrow \mathfrak{R}'$ satisfies
$$\Phi([A,B]) = [\Phi(A),\Phi(B)]$$
for all $A,B \in \mathfrak{R}$. Then $\Phi(A + B) = \Phi(A) + \Phi(B) + Z_{A,B}'$ for all $A,B \in \mathfrak{R}$, where $Z_{A,B}'$ is an element in the commutative centre $\mathcal{Z}(\mathfrak{R}')$ of $\mathfrak{R}'$ depending on $A$ and $B$.
\end{theorem}

In \cite{posd}, Ferreira and Guzzo investigated the additivity of Lie multiplicative map. They obtained the
following result. 

\begin{theorem}\label{FerGur} Let $\mathfrak{R}$ and $\mathfrak{R}'$ be alternative rings.
Suppose that $\mathfrak{R}$ is a ring containing a nontrivial idempotent $e_1$ which satisfies:
\begin{enumerate}
\item[\it (i)] If $[a_{11}+ a_{22}, \mathfrak{R}_{12}] = 0$, then $a_{11} + a_{22} \in \mathcal{Z}(\mathfrak{R}),$
\item[\it (ii)] If $[a_{11}+ a_{22}, \mathfrak{R}_{21}] = 0$, then $a_{11} + a_{22} \in \mathcal{Z}(\mathfrak{R}).$
\end{enumerate}
Then every Lie multiplicative bijection $\varphi$ of $\mathfrak{R}$ onto an arbitrary alternative ring $\mathfrak{R}'$ is almost additive.
\end{theorem}

In a recent paper, Ferreira and Guzzo studied the characterization of multiplicative Lie derivation 
on alternative rings, see \cite{FerGur}. They obtained the following result.

\begin{theorem} \label{Fegu2}
Let $\R$ be a unital $2$,$3$-torsion free alternative ring with nontrivial idempotents $e_1$, $e_2$ and with associated 
Peirce decomposition $\R = \R_{11} \oplus \R_{12} \oplus \R_{21} \oplus \R_{22}$. Suppose that $\mathfrak{R}$ satisfies the following conditions:
\begin{enumerate}
\item [{\rm (1)}] If $x_{ij}\R_{ji} = 0$, then $x_{ij} = 0$ $(i \neq j);$
\item [{\rm (2)}] If $x_{11}\R_{12} = 0$ or $\R_{21}x_{11} = 0$, then $x_{11} = 0;$
\item [{\rm (3)}] If $\R_{12}x_{22} = 0$ or $x_{22}\R_{21} = 0$, then $x_{22} = 0;$
\item [{\rm (4)}] If $z \in \mathcal{Z}(\mathfrak{R})$ with $z \neq 0$, then $z\R = \R.$
\end{enumerate}
Let $\D \colon  \R \longrightarrow \R$ be a multiplicative Lie derivation of $\mathfrak{R}$. 
Then $\D$ is the form $\delta + \tau$, where $\delta$ is an additive derivation of $\R$ and $\tau$ is a mapping from 
$\R$ into the commutative centre $\mathcal{Z}(\mathfrak{R})$,  which maps commutators into the zero if and only if
\begin{enumerate}
\item[(a)] $e_2\D(\R_{11})e_2 \subseteq \mathcal{Z}(\R) e_2,$
\item[(b)] $e_1\D(\R_{22})e_1 \subseteq \mathcal{Z}(\R) e_1.$
\end{enumerate}
\end{theorem}

Inspired by the above-mentioned results, we are planning to give a result about Lie multiplicative maps on alternative rings.

\begin{remark}
Note that prime alternative rings satisfy (1), (2), (3). 
\end{remark}

\begin{proposition}\label{prop4}
Let $\R$, $\R'$ be $2$-torsion free alternative ring. If $\varphi: \R \rightarrow \R'$ is a surjective Lie multiplicative map that preserves idempotents then $\varphi$ is injective and $\varphi(\lambda r) = \lambda \varphi(r)$ for every $\lambda \in \mathbb{Q}$ and $r \in \R$.
\end{proposition}
\begin{proof}
In first we will prove the injectivity. Suppose $\varphi(r) = \varphi(s)$ for some $r, s \in \R$. Since $\varphi(r) - \varphi(s)$ is idempotent in $\R'$ then $r-s$ is an idempotent in $\R$. In the same way $\varphi(s) - \varphi(r)$ is idempotent in $\R'$ and therefore also $s-r$ is an idempotent in $\R$. Since $r - s$
and $s - r$ are both idempotents in $\R$ it follows that $r-s = (r - s)^2 = s - r$, so $r = s$ and $\varphi$ is injective.
Now we will prove $\varphi(\lambda r) = \lambda \varphi(r)$ for every $\lambda \in \mathbb{Q}$ and $r \in \R$. Let $r \in \R$ and let $\lambda \in \mathbb{Q}$ with $\lambda \neq 0, -1$. Then $(\lambda r) - \lambda r \in  I(\R)$ and therefore $s = \varphi(\lambda r) - \lambda \varphi(r) \in I(\R')$. Similarly $r - (\frac{1}{\lambda})(\lambda r) \in I(\R)$ and so also $(\frac{-1}{\lambda})s = \varphi(r) - (\frac{1}{\lambda})\varphi(\lambda r) \in I(\R')$.
It follows that $-\frac{1}{\lambda}s = \frac{1}{\lambda^2}s^2 = \frac{1}{\lambda^2}s$ and so $s\left(1 + \frac{1}{\lambda}\right) = 0$. Therefore $\varphi(\lambda r) = \lambda \varphi(r)$ with $\lambda \neq 0, -1$. To $\lambda = -1$ we have $\varphi(-r) = -2\varphi(\frac{1}{2}r) = -\varphi(r)$. And to $\lambda = 0$ follows from
$\varphi(0) = \varphi([0,0]) = [\varphi(0), \varphi(0)] = 0$.
\end{proof}

\begin{proposition}\label{prop2}
Let $\R$ be a $2,3$-torsion free alternative ring satisfying $(1)$, $(2)$, $(3)$.
\begin{enumerate}
\item [$(\spadesuit)$] If $[a_{11}+ a_{22}, \mathfrak{R}_{12}] = 0,$ then $a_{11} + a_{22} \in \mathcal{Z}(\mathfrak{R}),$
\item [$(\clubsuit)$] If $[a_{11}+ a_{22}, \mathfrak{R}_{21}] = 0,$ then $a_{11} + a_{22} \in \mathcal{Z}(\mathfrak{R})$.
\end{enumerate}
\end{proposition}

\begin{proposition}\label{prop3}
If $\mathcal{Z}(\R_{ij}) = \left\{a \in \R_{ij} ~ | ~ [a, \R_{ij}] = 0 \right\}$ then $\mathcal{Z}(\R_{ij}) \subseteq \R_{ij} + \mathcal{Z}(\R)$ with $i \neq j$.
\end{proposition}

The reader can find the proof of these results in \cite{FerGur}.

\begin{remark}\label{obs1}
Let $\R$ be a $2$,$3$-torsion free alternative ring, $\R'$ another alternative ring and $\varphi: \R \rightarrow \R'$ be a surjective Lie multiplicative map that preserves idempotents. Note that $\varphi(e_1) = f_1$ is a nontrivial idempotent   in $\R'$ because $\varphi$  is a bijective map that preserves idempotents. Therefore $\R'$ has a Peirce decomposition $\R' = \R'_{11} \oplus \R'_{12} \oplus \R'_{21} \oplus \R'_{22}$ associated to the nontrivial idempotent   $f_1$.   
\end{remark}

\section{Main theorem}

We shall prove as follows the main result of this paper.

\begin{theorem}\label{mainthm} 
Let $\R$ be a unital $2$,$3$-torsion free alternative ring, $\R'$ another alternative ring and $\varphi: \R \rightarrow \R'$ be a surjective Lie multiplicative map that preserves idempotents. Assume that $\R$ has a nontrivial idempotent $e_1$ with associated Peirce decomposition $\R = \R_{11} \oplus \R_{12} \oplus \R_{21} \oplus \R_{22}$, such that 
\begin{enumerate}
\item [\it (1)] If $x_{ij}\R_{ji} = 0$ then $x_{ij} = 0$ $(i \neq j);$
\item [\it (2)] If $x_{11}\R_{12} = 0$ or $\R_{21}x_{11} = 0$ then $x_{11} = 0;$
\item [\it (3)] If $\R_{12}x_{22} = 0$ or $x_{22}\R_{21} = 0$ then $x_{22} = 0;$
\item [\it (4)] If $z \in \mathcal{Z}(\R)$ with $z \neq 0$ then $z\R = \R$.
\end{enumerate}
Then $\varphi$ is the form $\psi + \tau$, where $\psi$ is an additive isomorphism of $\R$ into $\R'$ and $\tau$ is a map from $\R$ into $\mathcal{Z}(\R')$, which maps commutators into the zero provided that
\begin{enumerate}
\item[$\left(\dagger\right)$] $f_i\varphi(\R_{jj})f_i \subseteq \mathcal{Z}(\R') f_i$
\end{enumerate}
or
$\varphi$ is the form $\psi + \tau$, where $\psi$ is a negative of an additive anti-isomorphism of $\R$ into $\R'$ and $\tau$ is a map from $\R$ into $\mathcal{Z}(\R')$, which maps commutators into the zero provided that
\begin{enumerate}
\item[$\left(\dagger \dagger \right)$] $f_i\varphi(\R_{ii})f_i \subseteq \mathcal{Z}(\R') f_i,$
\end{enumerate}
where $f_i = \varphi(e_i)$ and $f_j = 1_{\R'} - f_i$, $i \neq j$.
\end{theorem}

\vspace{.1in}
The following Lemmas has the same hypotheses of Theorem \ref{mainthm} and we need these Lemmas for the proof of this Theorem. 

\begin{lemma}\label{lema2}
Let $i, j \in \left\{1, 2\right\}$ with $i \neq j$. Then $\varphi(\R_{ij}) = \R'_{ij}$.
\end{lemma}
\begin{proof}
We show just the case $i = 1$ and $j=2$ because the other case is similar. Let $a_{12} \in \R_{12}$. It follows from the facts $a_{12} = [e_1, a_{12}]$ that 
$$\varphi(a_{12}) = [\varphi(e_1), \varphi(a_{12})] = \varphi(a_{12})_{12} - \varphi(a_{12})_{21}.$$
This implies that $\varphi(a_{12}) = \varphi(a_{12})_{12}$ for all $a_{12} \in \R_{12}$. Thus $\varphi(a_{12}) \in \R'_{12}$. Applying the same argument to $\varphi^{-1}$, we can obtain the reverse inclusion and equality follows. 
\end{proof}

\begin{lemma}\label{lema4}
$\varphi$ is an almost additive map, that is, for every $a,b \in \R$, $\varphi(a+b) - \varphi(a) - \varphi(b) \in \mathcal{Z}(\R')$.
\end{lemma}
\begin{proof}
Since $\R$ be a $2,3$-torsion free alternative ring satisfying $(1)$, $(2)$, $(3)$, $\R$ satisfies 
$(\spadesuit)$ and $(\clubsuit)$ by Proposition \ref{prop2}. Now using Theorem \ref{FerGur} we get $\varphi$ is an almost additive map. 
\end{proof}

\begin{remark}\label{impobs}
The assumptions $(2)$ and $(3)$ imply 
\begin{enumerate}
\item [\it (2')] If $x'_{11}\R'_{12} = 0$ or $\R'_{21}x'_{11} = 0$ then $x'_{11} = 0;$
\item [\it (3')] If $\R'_{12}x'_{22} = 0$ or $x'_{22}\R'_{21} = 0$ then $x'_{22} = 0;$
\end{enumerate}
regarding Peirce decomposition relative to idempotent $f_1$. Indeed just use the bijection of $\varphi$. 
\end{remark}

\subsection{First part of Theorem \ref{mainthm}}

Throughout this subsection we assume that $(\dagger)$ holds and let us consider $e_{1}$ a nontrivial idempotent of $\mathfrak{R}$.

\begin{lemma}\label{lema3} $\varphi(\R_{ii}) \subseteq \R'_{ii} + \mathcal{Z}(\R') \ (i = 1,2)$
\end{lemma}
\begin{proof}
We show just the case $i = 1$ because the other case can be treated similarly. For every $a_{11} \in \R_{11}$, with $\varphi(a_{11}) = b_{11} + b_{12} + b_{21} + b_{22}$ we get
$$0 = \varphi([a_{11}, e_1]) = [\varphi(a_{11}),f_1].$$
From this $b_{12} = b_{21} = 0$. By Theorem \ref{mainthm} item $(\dagger)$, we have
$$\varphi(a_{11}) = b_{11} + f_2\varphi(a_{11})f_2 = b_{11} + zf_2 = b_{11} -f_1z + z \in \R'_{11} + \mathcal{Z}(\R').$$
\end{proof}

\vspace{.1in}

Now let us define the mappings $\psi$ and $\tau$. By Lemmas \ref{lema2} and \ref{lema3} we have that

\begin{enumerate}
\item[\it (A)] if $a_{ij} \in \R_{ij}$, $i \neq j$, then $\varphi(a_{ij}) = b_{ij} \in \R'_{ij}$,
\item[\it (B)] if $a_{ii} \in \R_{ii}$, then $\varphi(a_{ii}) = b_{ii} + z, b_{ii} \in \R'_{ii}$, $z \in \mathcal{Z}(\R')$.
 \end{enumerate} 
We note that in $(B)$, $b_{ii}$ and $z$ are uniquely determined.  Now we define a map $\psi$ of $\R$ into $\R'$ according to the rule $\psi(a_{ij}) = b_{ij}, a_{ij} \in \R_{ij}$. For every $a = a_{11} + a_{12} + a_{21} + a_{22} \in \R$, define $\psi(a) = \sum \psi(a_{ij})$. A map $\tau$ of $\R$ into $\mathcal{Z}(\R')$ is then defined by
\begin{eqnarray*}
\tau(a) &=& \varphi(a) - \psi(a) =\\
&& \varphi(a) - (\psi(a_{11}) + \psi(a_{12}) + \psi(a_{21}) + \psi(a_{22})) =\\
&& \varphi(a) - (b_{11} + b_{12} + b_{21} + b_{22}) =\\
&& \varphi(a) - (\varphi(a_{11}) - z_{a_{11}} + \varphi(a_{12}) + \varphi(a_{21}) + \varphi(a_{22}) - z_{a_{22}}) =\\
&& 
\varphi(a) - (\varphi(a_{11}) + \varphi(a_{12}) + \varphi(a_{21}) + \varphi(a_{22}))  + (z_{a_{11}} + z_{a_{22}}). 
\end{eqnarray*}
We remark that $\psi(x) \in \mathcal{Z}(\R')$, if and only if $x \in \mathcal{Z}(\R)$. Now we need to prove that $\psi$ and $\tau$ are desired maps. 

\begin{lemma}\label{lema5}
$\psi$ is an additive map.
\end{lemma}
\begin{proof}
We only need to show that $\psi$ is an additive on $\R_{ii}$ because by Lemma $3.3$ in \cite{posd} we have $\varphi(a_{ij} + b_{ij}) = \varphi(a_{ij}) + \varphi(b_{ij})$, $i \neq j$. Let $a_{ii}, b_{ii} \in \R_{ii}$,
\begin{eqnarray*}
\psi(a_{ii} + b_{ii}) - \psi(a_{ii}) - \psi(b_{ii}) &=& \varphi(a_{ii} + b_{ii}) - \tau(a_{ii} + b_{ii}) - \varphi(a_{ii}) \\&+& \tau(a_{ii}) - \varphi(b_{ii}) + \tau(b_{ii}).
\end{eqnarray*} 
Thus, $\psi(a_{ii} + b_{ii}) - \psi(a_{ii}) - \psi(b_{ii}) \in \mathcal{Z}(\R') \cap \R'_{ii} = \left\{0\right\}$.
\end{proof}
Now we show that $\psi(ab) = \psi(a)\psi(b)$ for all $a, b \in \R$.

\begin{lemma}\label{lema6}
For every $a_{ii}, b_{ii} \in \R_{ii}$, $a_{ij}, b_{ij} \in \R_{ij}$, $b_{ji} \in \R_{ji}$ and $b_{jj} \in \R_{jj}$ with $i \neq j$ we have
\begin{enumerate}
\item[\it (I)] $\psi(a_{ii}b_{ij}) = \psi(a_{ii})\psi(b_{ij})$,
\item[\it (II)] $\psi(a_{ij}b_{jj}) = \psi(a_{ij})\psi(b_{jj})$,
\item[\it (III)] $\psi(a_{ii}b_{ii}) = \psi(a_{ii})\psi(b_{ii})$,
\item[\it (IV)] $\psi(a_{ij}b_{ij}) = \psi(a_{ij})\psi(b_{ij})$,
\item[\it (V)] $\psi(a_{ij}b_{ji}) = \psi(a_{ij})\psi(b_{ji}).$
\end{enumerate}
\end{lemma}
\begin{proof}
Let us start with $(I)$
\begin{eqnarray*}
\psi(a_{ii}b_{ij}) &=& \varphi(a_{ii}b_{ij})= \varphi([a_{ii}, b_{ij}])=[\varphi(a_{ii}), \varphi(b_{ij})]= \\
&& [\psi(a_{ii}), \psi(b_{ij})] = \psi(a_{ii})\psi(b_{ij}).
\end{eqnarray*} 
Next $(II)$
\begin{eqnarray*}
\psi(a_{ij}b_{jj}) &=& \varphi(a_{ij}b_{jj})= \varphi([a_{ij}, b_{jj}])=[\varphi(a_{ij}), \varphi(b_{jj})]= \\&& [\psi(a_{ij}), \psi(b_{jj})]= \psi(a_{ij})\psi(b_{jj}).
\end{eqnarray*} 
Now we show $(III)$. By $(I)$ we get
$$\psi((a_{ii}b_{ii})r_{ij}) = \psi(a_{ii}b_{ii})\psi(r_{ij}).$$ 
On the other hand,
\begin{eqnarray*}
\psi(a_{ii}(b_{ii}r_{ij})) &=& \psi(a_{ii})\psi(b_{ii}r_{ij}) = \psi(a_{ii})(\psi(b_{ii})\psi(r_{ij})).
\end{eqnarray*}
As $(a_{ii}b_{ii})r_{ij} = a_{ii}(b_{ii}r_{ij})$ and $(\psi(a_{ii})\psi(b_{ii}))\psi(r_{ij}) = \psi(a_{ii})(\psi(b_{ii})\psi(r_{ij}))$ we obtain
$$(\psi(a_{ii}b_{ii}) - \psi(a_{ii})\psi(b_{ii}))\psi(r_{ij}) = 0$$
for all $\psi(r_{ij}) \in \R'_{ij}$. So $\psi(a_{ii}b_{ii}) = \psi(a_{ii})\psi(b_{ii})$ by Remark \ref{impobs}.

Next ($IV$). 
\begin{eqnarray*}
2\psi(a_{ij}b_{ij}) &=& \psi(2a_{ij}b_{ij}) = \varphi(2a_{ij}b_{ij}) =\varphi([a_{ij},b_{ij}]) =  [\varphi(a_{ij}), \varphi(b_{ij})] = \\&& [\psi(a_{ij}), \psi(b_{ij})] = \psi(a_{ij})\psi(b_{ij}) - \psi(b_{ij})\psi(a_{ij}) = 2\psi(a_{ij})\psi(b_{ij}) 
\end{eqnarray*}
As $\R'$ is $2$- torsion free it is follow that $\psi(a_{ij}b_{ij}) = \psi(a_{ij})\psi(b_{ij})$. 
And finally we show $(V)$. We have
\begin{eqnarray*}
\tau([a_{ij}, b_{ji}]) &=& \varphi([a_{ij}, b_{ji}]) - \psi([a_{ij}, b_{ji}]) = [\varphi(a_{ij}), \varphi(b_{ji})] - \psi(a_{ij}b_{ji} - b_{ji}a_{ij})= \\&& [\psi(a_{ij}), \psi(b_{ji})] - \psi(a_{ij} b_{ji}) + \psi(b_{ji}a_{ij}) = \\&& 
\psi(a_{ij})\psi(b_{ji}) - \psi(b_{ji})\psi(a_{ij}) - \psi(a_{ij}b_{ji}) + \psi(b_{ji}a_{ij}), 
\end{eqnarray*}
which implies 
$$[\psi(a_{ij})\psi(b_{ji})- \psi(a_{ij}b_{ji}) ] + [\psi(b_{ji}a_{ij})- \psi(b_{ji})\psi(a_{ij})] = z' \in \mathcal{Z}(\R').$$
If $z'= 0$ then $\psi(a_{ij}b_{ji}) = \psi(a_{ij})\psi(b_{ji}).$
If $z' \neq 0$ we multiply by $\psi(a_{ij})$ we get
$$\psi(a_{ij})\psi(b_{ji}a_{ij}) - \psi(a_{ij})\psi(b_{ji})\psi(a_{ij}) = \psi(a_{ij})z'.$$
By $(II)$ we have
\begin{eqnarray}\label{dif}
\psi(a_{ij}b_{ji}a_{ij}) - \psi(a_{ij})\psi(b_{ji})\psi(a_{ij}) = \psi(a_{ij})z'.
\end{eqnarray}
Now we observe that $\psi(a_{ij}b_{ji}a_{ij}) = \psi(a_{ij})\psi(b_{ji})\psi(a_{ij})$. In deed, observe that $[[a_{ij}, b_{ji}],a_{ij}] = 2a_{ij}b_{ji}a_{ij}$. Then
\begin{eqnarray*}
2\psi(a_{ij}b_{ji}a_{ij}) &=& \psi(2a_{ij}b_{ji}a_{ij}) = \varphi([[a_{ij}, b_{ji}],a_{ij}])=
 [[\varphi(a_{ij}), \varphi(b_{ji})],\varphi(a_{ij})] = \\&&
 [[\psi(a_{ij}), \psi(b_{ji})],\psi(a_{ij})]= 2 \psi(a_{ij})\psi(b_{ji})\psi(a_{ij})
\end{eqnarray*}
Since $\R'$ is $2$-torsion free we get $\psi(a_{ij}b_{ji}a_{ij}) = \psi(a_{ij})\psi(b_{ji})\psi(a_{ij})$.
So $\psi(a_{ij})z' = 0$ that implies $a_{ij}z = 0$ with $z \in \mathcal{Z}(\R).$ 
Now, we will consider the case $i=1$ and $j = 2,$ the other case is similar.
By $(4)$  there exists $h \in \R$ such that $zh = e_1 + e_2.$
Hence,  
\begin{eqnarray*}a_{12}(z_{22}h) &=& a_{12}(z_{22}h_{21}) + a_{12}(z_{22}h_{22}) = \\
&&(a_{12}z_{22})h_{21} + (a_{12}z_{22})h_{22} = (a_{12}z)h_{21} + (a_{12}z)h_{22} = 0.
\end{eqnarray*}
Now, 
$z_{22}h = 1_{\R} - z_{11}h$ gives 
$a_{12}(z_{22}h) = a_{12} - a_{12}(z_{11}h),$ and follows that $a_{12} = 0.$
Summarizing with the other case, we have
 $a_{ij} = 0$ which is a contradiction. Therefore $\psi(a_{ij}b_{ji}) = \psi(a_{ij})\psi(b_{ji}).$
\end{proof}

\begin{lemma}\label{lema7}
$\psi$ is a homomorphism.
\end{lemma}

\begin{lemma}\label{lema8}
$\tau$ sends the commutators into zero.
\end{lemma}
\begin{proof}
\begin{eqnarray*}
\tau([a, b]) &=& \varphi([a, b]) - \psi([a, b])= [\varphi(a),\varphi(b)] - \psi([a, b])=\\
&& [\psi(a) ,\psi(b)] - \psi([a, b])= 0.
\end{eqnarray*}
\end{proof}
The first part of proof of the Theorem \ref{mainthm} is completed.

\subsection{Second part of Theorem \ref{mainthm}}

Now throughout this subsection we assume that $(\dagger \dagger)$ holds and also let us consider $e_{1}$ a nontrivial idempotent of $\mathfrak{R}$.

\begin{lemma}\label{lema32} $\varphi(\R_{ii}) \subseteq \R'_{jj} + \mathcal{Z}(\R') \ (i \neq j)$
\end{lemma}
\begin{proof}
We show just the case $i = 1$ and $j =2$ because the other case can be treated similarly. For every $a_{11} \in \R_{11}$, with $\varphi(a_{11}) = b_{11} + b_{12} + b_{21} + b_{22}$ we get
$$0 = \varphi([a_{11}, e_1]) = [\varphi(a_{11}),f_1].$$
From this $b_{12} = b_{21} = 0$. By Theorem \ref{mainthm} item $(\dagger \dagger)$, we have
$$\varphi(a_{11}) = f_1\varphi(a_{11})f_1 + b_{22} = zf_1 + b_{22} = b_{22} -zf_2 + z \in \R'_{22} + \mathcal{Z}(\R').$$
\end{proof}

Let us define the mappings $\psi$ and $\tau$. By Lemmas \ref{lema2} and \ref{lema32} we have that

\begin{enumerate}
\item[\it $(A')$] if $a_{ij} \in \R_{ij}$, $i \neq j$, then $\varphi(a_{ij}) = b_{ij} \in \R'_{ij}$,
\item[\it $(B')$] if $a_{ii} \in \R_{ii}$, then $\varphi(a_{ii}) = b_{jj} + z, b_{jj} \in \R'_{jj}$, $z \in \mathcal{Z}(\R')$.
 \end{enumerate} 
We again observe that in $(B')$, $b_{jj}$ and $z$ are uniquely determined. Now we define a map $\psi$ of $\R$ into $\R'$ according to the rule $\psi(a_{ij}) = b_{ij}, a_{ij} \in \R_{ij}$ and $\psi(a_{ii}) = b_{jj}$, $a_{ii} \in \R_{ii}$ with $i \neq j$. For every $a = a_{11} + a_{12} + a_{21} + a_{22} \in \R$, define $\psi(a) = \sum \psi(a_{ij})$. A map $\tau$ of $\R$ into $\mathcal{Z}(\R')$ is then defined by
\begin{eqnarray*}
\tau(a) = \varphi(a) - \psi(a).
\end{eqnarray*}
We again remark that $\psi(x) \in \mathcal{Z}(\R')$, if and only if $x \in \mathcal{Z}(\R)$. Now we need to prove that $\psi$ and $\tau$ are desired maps.

\begin{lemma}\label{lema52}
$\psi$ is an additive map.
\end{lemma}
\begin{proof}
Similar to proof of Lemma \ref{lema5}.
\end{proof}

\vspace{0.2in}

Now we show that $\psi(ab) = -\psi(b)\psi(a)$ for all $a, b \in \R$. For this purpose we prove the following Lemma whose demonstration is similar to Lemma \ref{lema6}, but for better clarity of the text we will proof it.

\begin{lemma}\label{lema62}
For every $a_{ii}, b_{ii} \in \R_{ii}$, $a_{ij}, b_{ij} \in \R_{ij}$, $b_{ji} \in \R_{ji}$ and $b_{jj} \in \R_{jj}$ with $i \neq j$ we have
\begin{enumerate}
\item[\it $(I')$] $\psi(a_{ii}b_{ij}) = -\psi(b_{ij})\psi(a_{ii})$,
\item[\it $(II')$] $\psi(a_{ij}b_{jj}) = -\psi(b_{jj})\psi(a_{ij})$,
\item[\it $(III')$] $\psi(a_{ii}b_{ii}) = -\psi(b_{ii})\psi(a_{ii})$,
\item[\it $(IV')$] $\psi(a_{ij}b_{ij}) = -\psi(b_{ij})\psi(a_{ij})$,
\item[\it $(V')$] $\psi(a_{ij}b_{ji}) = -\psi(b_{ji})\psi(a_{ij}).$
\end{enumerate}
\end{lemma}
\begin{proof}
Let us start with $(I')$
\begin{eqnarray*}
\psi(a_{ii}b_{ij}) &=& \varphi(a_{ii}b_{ij})= \varphi([a_{ii}, b_{ij}]) = [\varphi(a_{ii}), \varphi(b_{ij})] =\\
&&[\psi(a_{ii}), \psi(b_{ij})] = -\psi(b_{ij})\psi(a_{ii}).
\end{eqnarray*} 

Next $(II')$
\begin{eqnarray*}
\psi(a_{ij}b_{jj}) &=& \varphi(a_{ij}b_{jj})= \varphi([a_{ij}, b_{jj}]) = [\varphi(a_{ij}), \varphi(b_{jj})] =\\
&& [\psi(a_{ij}), \psi(b_{jj})] = -\psi(b_{jj})\psi(a_{ij}).
\end{eqnarray*} 
Now we show $(III')$. By $(I)$ we get
$$\psi((a_{ii}b_{ii})r_{ij}) = -\psi(r_{ij})\psi(a_{ii}b_{ii}).$$ 
On the other hand,
\begin{eqnarray*}
\psi(a_{ii}(b_{ii}r_{ij})) &=& -\psi(b_{ii}r_{ij})\psi(a_{ii}) =\\&&
 -(-\psi(r_{ij})\psi(b_{ii}))\psi(a_{ii}) = \psi(r_{ij})(\psi(b_{ii})\psi(a_{ii})).
\end{eqnarray*}
As $(a_{ii}b_{ii})r_{ij} = a_{ii}(b_{ii}r_{ij})$ and $-\psi(r_{ij})\psi(a_{ii}b_{ii}) = \psi(r_{ij})(\psi(b_{ii})\psi(a_{ii}))$ we obtain
$$\psi(r_{ij})(\psi(a_{ii}b_{ii}) + \psi(b_{ii})\psi(a_{ii})) = 0$$
for all $\psi(r_{ij}) \in \R'_{ij}$. So $\psi(a_{ii}b_{ii}) = -\psi(b_{ii})\psi(a_{ii})$ by Remark \ref{impobs}.

Next ($IV'$). 
\begin{eqnarray*}
2\psi(a_{ij}b_{ij}) &=& \psi(2a_{ij}b_{ij}) = \varphi(2a_{ij}b_{ij}) = \varphi([a_{ij},b_{ij}]) = [\varphi(a_{ij}), \varphi(b_{ij})] = \\&& 
[\psi(a_{ij}), \psi(b_{ij})] = \psi(a_{ij})\psi(b_{ij}) - \psi(b_{ij})\psi(a_{ij}) = -2\psi(b_{ij})\psi(a_{ij}) 
\end{eqnarray*}
As $\R'$ is $2$- torsion free it is follow that $\psi(a_{ij}b_{ij}) = -\psi(b_{ij})\psi(a_{ij})$. 

And finally we show $(V')$. We have
\begin{eqnarray*}
\tau([a_{ij}, b_{ji}]) &=& \varphi([a_{ij}, b_{ji}]) - \psi([a_{ij}, b_{ji}]) = [\varphi(a_{ij}), \varphi(b_{ji})] - \psi(a_{ij}b_{ji} - b_{ji}a_{ij}) =\\
&& [\psi(a_{ij}), \psi(b_{ji})] - \psi(a_{ij} b_{ji}) + \psi(b_{ji}a_{ij}) =\\&& 
\psi(a_{ij})\psi(b_{ji}) - \psi(b_{ji})\psi(a_{ij}) - \psi(a_{ij}b_{ji}) + \psi(b_{ji}a_{ij}), 
\end{eqnarray*}
which implies 
$$[- \psi(a_{ij}b_{ji})- \psi(b_{ji})\psi(a_{ij}) ] + [\psi(a_{ij})\psi(b_{ji}) + \psi(b_{ji}a_{ij})] = z' \in \mathcal{Z}(\R').$$
If $z'= 0$ then $\psi(a_{ij}b_{ji}) = -\psi(b_{ji})\psi(a_{ij}).$
If $z' \neq 0$ we multiply by $\psi(a_{ij})$ we get
$$\psi(a_{ij})\psi(b_{ji})\psi(a_{ij}) + \psi(b_{ji}a_{ij})\psi(a_{ij}) = \psi(a_{ij})z'.$$
By $(II')$ we have
\begin{eqnarray}\label{dif}
-\psi(a_{ij}b_{ji}a_{ij}) + \psi(a_{ij})\psi(b_{ji})\psi(a_{ij}) = \psi(a_{ij})z'.
\end{eqnarray}
Now we observe that $\psi(a_{ij}b_{ji}a_{ij}) = \psi(a_{ij})\psi(b_{ji})\psi(a_{ij})$. In deed, observe that $[[a_{ij}, b_{ji}],a_{ij}] = 2a_{ij}b_{ji}a_{ij}$. Then
\begin{eqnarray*}
2\psi(a_{ij}b_{ji}a_{ij}) &=& \psi(2a_{ij}b_{ji}a_{ij}) = \varphi([[a_{ij}, b_{ji}],a_{ij}]) =[[\varphi(a_{ij}), \varphi(b_{ji})],\varphi(a_{ij})] =\\ &&
  [[\psi(a_{ij}), \psi(b_{ji})],\psi(a_{ij})] = 2 \psi(a_{ij})\psi(b_{ji})\psi(a_{ij})
\end{eqnarray*}
Since $\R'$ is $2$-torsion free we get $\psi(a_{ij}b_{ji}a_{ij}) = \psi(a_{ij})\psi(b_{ji})\psi(a_{ij})$.
So $\psi(a_{ij})z' = 0$ that implies $a_{ij}z = 0$ with $z \in \mathcal{Z}(\R)$ but by $(4)$ there exist $h \in R$ such that $zh = e_1 + e_2$ and $a_{ij} = 0$ which is a contradiction. Therefore $\psi(a_{ij}b_{ji}) = -\psi(b_{ji})\psi(a_{ij}).$
   
\end{proof}

\begin{lemma}\label{lema72}
$\psi$ is an negative of an homomorphism.
\end{lemma}

\begin{lemma}\label{lema82}
$\tau$ sends the commutators into zero.
\end{lemma}
\begin{proof}
The same proof of Lemma \ref{lema8}.
\end{proof}

\vspace{.2in}

The second part of proof of the Theorem \ref{mainthm} is completed.

\vspace{0,5cm}

\section{Applications}

\begin{corollary}
Let $\R$ be a unital $2$,$3$-torsion free prime alternative ring, $\R'$ another prime alternative ring and $\varphi: \R \rightarrow \R'$ be a surjective Lie multiplicative map that preserves idempotents. Assume that $\R$ has a nontrivial idempotent $e_1$ with associated Peirce decomposition $\R = \R_{11} \oplus \R_{12} \oplus \R_{21} \oplus \R_{22}$ satisfing $(4)$ of Theorem \ref{mainthm}.
Then $\varphi$ is the form $\psi + \tau$, where $\psi$ is an additive isomorphism of $\R$ into $\R'$ and $\tau$ is a map from $\R$ into $\mathcal{Z}(\R')$, which maps commutators into the zero provided that
\begin{enumerate}
\item[$\left(\dagger\right)$] $f_i\varphi(\R_{jj})f_i \subseteq \mathcal{Z}(\R') f_i$
\end{enumerate}
or
$\varphi$ is the form $\psi + \tau$, where $\psi$ is a negative of an additive anti-isomorphism of $\R$ into $\R'$ and $\tau$ is a map from $\R$ into $\mathcal{Z}(\R')$, which maps commutators into the zero provided that
\begin{enumerate}
\item[$\left(\dagger \dagger \right)$] $f_i\varphi(\R_{ii})f_i \subseteq \mathcal{Z}(\R') f_i,$
\end{enumerate}
where $f_i = \varphi(e_i)$ and $f_j = 1_{\R'} - f_i$, $i \neq j$.
\end{corollary}

We provide an application on simple alternative rings.

\begin{corollary}
Let $\R$ be a unital $2$,$3$-torsion free simple alternative ring, $\R'$ another simple alternative ring and $\varphi: \R \rightarrow \R'$ be a surjective Lie multiplicative map that preserves idempotents. Assume that $\R$ has a nontrivial idempotent $e_1$ with associated Peirce decomposition $\R = \R_{11} \oplus \R_{12} \oplus \R_{21} \oplus \R_{22}$.
Then $\varphi$ is the form $\psi + \tau$, where $\psi$ is an additive isomorphism of $\R$ into $\R'$ and $\tau$ is a map from $\R$ into $\mathcal{Z}(\R')$, which maps commutators into the zero provided that
\begin{enumerate}
\item[$\left(\dagger\right)$] $f_i\varphi(\R_{jj})f_i \subseteq \mathcal{Z}(\R') f_i$
\end{enumerate}
or
$\varphi$ is the form $\psi + \tau$, where $\psi$ is a negative of an additive anti-isomorphism of $\R$ into $\R'$ and $\tau$ is a map from $\R$ into $\mathcal{Z}(\R')$, which maps commutators into the zero provided that
\begin{enumerate}
\item[$\left(\dagger \dagger \right)$] $f_i\varphi(\R_{ii})f_i \subseteq \mathcal{Z}(\R') f_i,$
\end{enumerate}
where $f_i = \varphi(e_i)$ and $f_j = 1_{\R'} - f_i$, $i \neq j$.
\end{corollary}

As a last application we provide a result in associative rings.

\begin{corollary}
Let $\R$ be a unital $2$,$3$-torsion free simple associative ring, $\R'$ another simple associative ring and $\varphi: \R \rightarrow \R'$ be a surjective Lie multiplicative map that preserves idempotents. Assume that $\R$ has a nontrivial idempotent. Then $\varphi$ is of the form $\psi + \tau$, where $\psi$ is either
an isomorphism or the negative of an anti-isomorphism of $\R$ onto $\R'$ and $\tau$ is an additive mapping of $\R$ into the centre of $\R'$ which maps commutators into zero.
\end{corollary}
\begin{proof}
By Theorem $3$ in \cite{Mart2} we have that either
\begin{eqnarray*}
(i) &&\R_{11} = \left\{a \in \R_{11} ~|~  \varphi(a) \in \R'_{11} + \mathcal{Z}(\R') \right\}  \\&& and \\&& \R_{22} = \left\{b \in \R_{22} ~|~  \varphi(b) \in \R'_{22} + \mathcal{Z}(\R') \right\}
	\end{eqnarray*}

or

\begin{eqnarray*}
(ii) &&\R_{11} = \left\{a \in \R_{11} ~|~  \varphi(a) \in \R'_{22} + \mathcal{Z}(\R') \right\}  \\&&and \\&& \R_{22} = \left\{b \in \R_{22} ~|~  \varphi(b) \in \R'_{11} + \mathcal{Z}(\R') \right\}.	
\end{eqnarray*}
If $(i)$ is satisfied then $(\dagger)$ of Theorem \ref{mainthm} is holds. But if $(ii)$ is satisfied then $(\dagger \dagger)$ of Theorem \ref{mainthm} is holds.
Therefore $\varphi$ is of the form $\psi + \tau$, where $\psi$ is either an isomorphism or the negative of an anti-isomorphism of $\R$ onto $\R'$ and $\tau$ is an additive mapping of $\R$ into the centre of $\R'$ which maps commutators into zero.
\end{proof}

\end{document}